\documentclass[draft,12pt]{article}

\usepackage[latin1]{inputenc}
\usepackage{amsmath,amssymb}
\usepackage{latexsym}
\usepackage{setspace}
\usepackage{amsthm}

\newtheorem{theorem}{Theorem}[section]
\newtheorem{corollary}[theorem]{Corollary}
\newtheorem{lemma}[theorem]{Lemma}
\newtheorem{proposition}[theorem]{Proposition}
\newtheorem{definition}[theorem]{Definition}
\newtheorem{assumption}[theorem]{Assumption}
\newtheorem{remark}[theorem]{Remark}

\numberwithin{equation}{section}

\def\square{{\vcenter{\vbox{\hrule height.3pt
        \hbox{\vrule width.3pt height5pt \kern5pt
           \vrule width.3pt}
        \hrule height.3pt}}}}

  \def\sF {{\cal F}}

 \def\bE {{\mathbb E}}

\def\bP {{\mathbb P}}  \def\bR {{\mathbb R}}

\def\P{{\mathbb P}}

\def\bee{\begin{equation}}
\def\bet{\begin{theorem}}
\def\bep{\begin{proposition}}
\def\bel{\begin{lemma}}
\def\bec{\begin{corollary}}
\def\bed{\begin{definition}}
\def\ber{\begin{remark}}
\def\eee{\end{equation}}
\def\eet{\end{theorem}}
\def\eep{\end{proposition}}
\def\eel{\end{lemma}}
\def\eec{\end{corollary}}
\def\eed{\end{definition}}
\def\eer{\end{remark}}

\def\R{{\mathbb R}}

\def\P{{\mathbb P}}

\def\ni{\noindent }
\def\ms{\medskip}

\def\square{{\vcenter{\vbox{\hrule height.3pt
        \hbox{\vrule width.3pt height5pt \kern5pt
           \vrule width.3pt}
        \hrule height.3pt}}}}

\def\tlint{{- \kern-0.85em \int \kern-0.2em}}  
\def\dlint{{- \kern-1.05em \int \kern-0.4em}}  

  \def\sF {{\cal F}}

 \def\bE {{\mathbb E}}

\def\bP {{\mathbb P}}  \def\bR {{\mathbb R}}

\def\nn{{\nonumber}}

\begin{document}
\title{Pathwise uniqueness of one-dimensional SDEs driven by one-sided stable processes }
\author {Hua Ren}
\date{\today}

\maketitle

\begin{abstract}
For $\alpha\in (0,1)$, we consider stochastic differential equations driven by one-sided stable processes of order $\alpha$:  \[dX_t= \phi(X_{t-})\ dZ_t.\] We prove that pathwise uniqueness holds for this equation under the assumptions that $\phi$ is continuous, non-decreasing and positive on $\R$. A counterexample is given to show that the positivity of $\phi$ is crucial for pathwise uniqueness to hold.

\end{abstract}

\section {Introduction}

The theory of stochastic differential equations (SDEs) provides a very important tool for constructing stochastic processes. A one-dimensional SDE driven by Brownian motion without a drift term is of the form
\begin{align}
dX_t= \sigma(X_t)\ dB_t,
\label{CE00}
\end{align}
where $B$ is a Brownian motion.

In the recent years, there has been intense interest in studying SDEs with jumps. One can get a jump type analogue of \eqref{CE00} by replacing the Brownian motion $B_t$ by a L\'evy process with jumps, for example, an $\alpha$-stable process $Z_t$:
\begin{align}
dX_t= \sigma(X_{t-})\ dZ_t.
\label{CE01}
\end{align}

For equation \eqref{CE00}, it is well-known that if the coefficients are assumed to be Lipschitz continuous, then the existence of strong solutions and pathwise uniqueness can be easily obtained by Picard interation and Gronwall's inequality. This condition has been improved by Yamada and Watanabe \cite{YW}, who showed that pathwise uniqueness holds when $\sigma$ is H\"older continuous of order $\frac12$.  For weak uniqueness, a result by Engelbert and Schmidt \cite{ES} says that if $\sigma^2(x)> 0$ on $\R$, then weak uniqueness holds for \eqref{CE00}. For jump SDEs, when $Z_t$ is a symmetric stable process of order $\alpha\in (1,2)$, Bass \cite{Bass1} proved pathwise uniqueness for equation \eqref{CE01} if $\int_{0+}\rho(x)^{-\alpha}\ dx= \infty$, where $\rho$ represents the modulus of continuity of $\sigma$. In particular, if $\sigma$ is H\"older continuous of order $\frac{1}{\alpha}$, then pathwise uniqueness holds for \eqref{CE01}. See also Komatsu \cite{Kom}. Bass, Burdzy and Chen \cite{BBC} proved the above condition is sharp, based on ideas from Barlow \cite{Bar}. Based on the result of \cite{BBC}, when $Z_t$ is a symmetric stable process of order $\alpha\in (0,1)$, one has to require $\sigma$ to be almost Lipschitz to have pathwise uniqueness for \eqref{CE01}.\\

Fournier \cite{Fou} weakened the conditions when the driving process is a stable process of order $\alpha\in (0,1)$ by studying the following equation:
\begin{align}
dX_t= \int_{\bR\setminus{\{0\}}}\int_{\bR\setminus{\{0\}}}z[1_{\{0<u<\gamma(X_{s-})\}}-1_{\{\gamma(X_{s-})< u< 0\}}]\ M(ds\ dz\ du),
\label{CE02}
\end{align}
where $\gamma(x)= sign(\sigma(x))\cdot|\sigma(x)|^{\alpha}$ and $M$ is a certain Poisson random measure. Refer to \cite{Fou} for more details. One of the results in \cite{Fou} says that if $\sigma$ is bounded below from zero, and $\sigma$ is H\"older continuous of order $\alpha$ if $\alpha\in (0,\frac12]$ and non-increasing and H\"older continuous of order $1-\alpha$ if $\alpha\in (\frac12, 1)$, then pathwise uniqueness holds for \eqref{CE02} when $Z_t$ is a one-sided stable process of order $\alpha$.\\

In this paper, we study the equation
\begin{align}
dX_t= \phi(X_{t-})\ dZ_t,
\label{EQ0}
\end{align}
where $Z_t$ is a one-sided stable process of order $\alpha\in(0,1)$. The characteristic function of $Z_t$ is given by
 \[\bE e^{iuZ_t}= exp \Big\{t\int_0^{\infty}(f(x+h)-f(x))\frac{c}{|h|^{d+\alpha}}\ dh\Big\},\]
where $c$ is a positive constant. In particular, $Z_t$ only has positive jumps. Our aim is to examine the conditions on $\phi$ under which pathwise uniqueness holds for \eqref{EQ0}. \\

The assumptions we put on $\phi$ are the following:
\begin{assumption}
\begin{enumerate}\label{ASS1}
 \item $\phi(\cdot)$ is continuous on $\mathbb{R}$;
 \item $\phi(\cdot)$ is non-decreasing on $\mathbb{R}$;
 \item $\phi(\cdot)$ is positive on $\mathbb{R}$.
\end{enumerate}
\end{assumption}

We prove that under Assumption \ref{ASS1}, pathwise uniqueness holds for equation \eqref{EQ0}.

\begin{theorem}
Suppose $\phi$ satisfies Assumption \ref{ASS1}. Then the solution to equation \eqref{EQ0} is pathwise unique.\\
\label{THM1}
\end{theorem}

To prove Theorem \ref{THM1}, our strategy is to first construct a strong solution $X_t$ to equation \eqref{EQ0} and then show that weak uniqueness holds for \eqref{EQ0}. Once we finish these two steps, pathwise uniqueness will follow.\\

For more background
information on jump SDEs, see the books \cite{App}, \cite{Sat}, \cite{Situ}, \cite{Sko} and a survey paper \cite{Bass2} by Bass .
See \cite{Bass1}, \cite{BBC}, \cite{BBO}, \cite{ES}, \cite{Fou},
\cite{KS},
\cite{LP}, \cite{PZ}, \cite{RW},
\cite{Zan1} and \cite{Zan2}
for more results on the existence and uniqueness for the solutions of closely related SDEs with jumps. \\


\section{Existence of the strong solution}

Let $(\Omega, \sF, \{\sF_t\}_{t\ge0},\P)$ be a filtered probability space which satisfies the usual conditions. For $\alpha\in(0,1)$, let $Z_t$ be a $\sF_t$-adapted one-sided $\alpha$-stable process. The goal of this section is to construct a strong solution to equation $\eqref{EQ0}$.\\

\begin{proposition}\label{PRO1'}
Suppose $\phi$ satisfies Assumption \ref{ASS1}. Then there exists a strong solution to equation \eqref{EQ0}.
\end{proposition}

\begin{proof} For each $n\in \mathbb{N}$, we define
\[Z_t^n= \sum_{s\le t}\Delta Z_s1_{\{\Delta Z_s\ge \frac{1}{n}\}}.\] Then $Z_t^n$ is adapted to $\sF_t$. Recall that $Z_t$ only has positive jumps and no continuous part, so for each $t\ge 0, \{Z_t^n\}_{n\ge 1}$ is a non-decreasing process and $Z_t^n\rightarrow Z_t$ $\mathbb{P}$-a.s. as $n\rightarrow\infty$.

For any $x_0\in \mathbb{R}$, let $X^n$ be the solution to
\begin{align*}
d X_t^n= \phi(X_{t-}^n)\ d Z_t^n, \ \ X_0^n= x_0.
\end{align*}
Recall that there are only finitely many jumps of $Z_t^n$ on a finite time interval, so it is easy to see the existence of the solution $X^n$ and that the solution is uniquely determined by the initial condition: $X_t^n$ will stay constant until the first jump of $Z_t^n$, at which point $X_t^n$ will jump $\phi(X_{t-}^n)\Delta Z_t$.\\

It is clear that $X_t^n$ is adapted to $\sF_t$ for each $n\in \mathbb{N}$. For each $t\ge 0$, we show that $\{X_t^n\}_{n\ge 1}$ is also a non-decreasing sequence.\\

Take $n, m\in \mathbb{N}$ with $n> m$. We claim that $X_t^{n}\ge X_t^m$ $\mathbb{P}$-a.s. for $t\ge 0$. If not, let \[S=\inf\{t\ge 0: X_t^n< X_t^m\}.\]
Then $X_{S-}^n\ge X_{S-}^m$. Since $X_S^n= X_{S-}^n+ \phi(X_{S-}^n)\Delta Z_S$ and $X_S^m= X_{S-}^m+ \phi(X_{S-}^m)\Delta Z_S$, remembering that $\phi$ is non-decreasing, then $\phi(X_{S-}^n)\ge \phi(X_{S-}^m)$. Therefore $X_S^n\ge X_S^m$. Clearly $S$ must be a jump time of $Z^n$ if $S< \infty$. But on the set $S< \infty$, $Z^n$ and $Z^m$, and therefore also $X^n$ and $X^m$, will be constant for a positive length of time after time $S$. But this contradicts the definition of $S$. We conclude that $S= \infty$ $\bP$-a.s.\\

Let $X_t= \lim_{n\rightarrow\infty} X_t^n$ and note that $X_t$ is adapted to $\sF_t$. We have
\[X_t^n= x_0+ \sum_{s\le t}\phi(X_{s-}^n)1_{\{\Delta Z_s\ge \frac{1}{n}\}}\Delta Z_s.\]
As $n\rightarrow\infty$, the right hand side converges to
\[x_0+ \sum_{s\le t}\phi(X_{s-})\Delta Z_s\] by monotone convergence. Since $Z$ is non-decreasing and has no continuous part, we conclude that
\[X_t= x_0+ \int_0^t\phi(X_{s-})\ dZ_s.\]

\end{proof}


\section{Weak uniqueness}
\begin{proposition}
Suppose $\phi$ satisfies Assumption \ref{ASS1}. Then the solution to equation \eqref{EQ0} is unique in law.\\
\label{PRO2}
\end{proposition}

\begin{proof} Let $X_t$ denote a weak solution to equation \eqref{EQ0}. Then by Theorem 4.1 in \cite{Kal}, for any real number $x$, there exists a process $\tilde{Z_t}$ such that $\tilde{Z}$ has the same law as $Z$ and $X_t= x+ \tilde{Z}_{\tau_t}$ where $\tau_t= \int_0^t\phi(X_s)^{\alpha}\ ds$.\\

Let $B_t$ be the right inverse of $\tau_t$, i.e.,
\begin{align}
B_t= \inf\{s\ge 0: \tau_s> t\}.
\label{CE03}
\end{align}

We will show that
\begin{align}
B_t= \int_0^t\phi(x+ \tilde{Z_s})^{-\alpha}\ ds
\label{CEE01}
\end{align}
and then conclude that $\tau_t= \inf\{s\ge 0: \int_0^s\phi(x+ \tilde{Z_u})^{-\alpha} du> t\}$ $\mathbb{P}$-a.s. Therefore the distribution of $X_t$ will be determined by the one of $\tilde{Z}$ for any given initial value $x$.\\

Recalling that $\phi$ is positive on $\bR$, for any $t\ge 0$, by Lemma 1.6 in \cite{ES}, we have
\begin{align}
B_t&= \int_0^{B_t}1_{\{\phi(X_s)\neq0\}}\ ds= \int_0^{t\wedge \tau_{\infty}}\phi(X_{B_s})^{-\alpha}\ ds\nn\\
&=\int_0^{t\wedge\tau_{\infty}}\phi (x+\tilde{Z_s})^{-\alpha}\ ds.
\label{CE06}
\end{align}

Case I: For $t\ge \tau_{\infty}$. Since $\phi$ is positive on $\bR$, then $\tau_t$ is strictly increasing on $[0, \infty]$. Then, by \eqref{CE03},
\begin{align*}
B_t&= \inf\{s\ge 0: \tau_s> t\}\ge \inf\{s\ge 0: \tau_s> \tau_{\infty}\}\\
&= \infty.
\end{align*}

Therefore, by \eqref{CE06}, $\int_0^{t\wedge \tau_{\infty}}\phi(x+\tilde{Z}_s)^{-\alpha}\ ds= B_t= \infty$. Since \[\int_0^t\phi(x+\tilde{Z}_s)^{-\alpha}\ ds\ge \int_0^{t\wedge \tau_{\infty}} \phi(x+\tilde{Z}_s)^{-\alpha}\ ds,\] then \[\int_0^t\phi(x+\tilde{Z}_s)^{-\alpha}\ ds= \infty.\]

Hence, $B_t= \infty= \int_0^t\phi(x+\tilde{Z}_s)^{-\alpha}\ ds$, i.e., \eqref{CEE01} holds with both sides equal to $\infty$.\\

Case II: For $t< \tau_{\infty}$. Then, by \eqref{CE06}, \[B_t= \int_0^{t\wedge\tau_{\infty}} \phi(x+\tilde{Z_s})^{-\alpha}\ ds= \int_0^t \phi(x+\tilde{Z_s})^{-\alpha}\ ds,\]
which is just \eqref{CEE01}.

Hence \eqref{CEE01} has been proved. Therefore we can conclude that for any given initial value $x$, the law of $X_t$ is unique.\\
\end{proof}


\section{Pathwise uniqueness}

Now we are ready to give the proof of Theorem \ref{THM1}.

\begin{proof}[\textnormal{\textbf{Proof of Theorem \ref{THM1}}}]
According to Proposition \ref{PRO1'}, there exists a strong solution $X_t$ to equation \eqref{EQ0}. Therefore there exists a measurable map $H: Z\mapsto X.$ Suppose $X_t'$ is another solution. Then by Proposition \ref{PRO2}, the laws of $X$ and $X'$ are the same. Since $Z_t= \int_0^t\phi(X_{s-})^{-1}\ dX_s$ and $Z_t= \int_0^t\phi(X_{s-}')^{-1}\ dX_s'$, then the joint laws of $(Z,X)$ and $(Z,X')$ are the same. Since $X= H(Z)$, then $X'= H(Z)$. Therefore, $X'= H(Z)= X$.\\
\end{proof}

\begin{remark}
We give a counterexample to show that the positivity condition on $\phi$ is crucial to obtaining pathwise uniqueness.\\

For $\alpha\in (0,1)$ and $\beta\in (0,1)$, let $Z_t$ be a one-sided $\alpha$-stable process which starts from zero. Let $\phi(x)=x^{\beta}$ for $x\ge 0$. Define $B_t= \int_0^t \phi(Z_{s-})^{-\alpha}\ ds$. Then from known estimates on the transition density function of stable processes, $B_t$ will be finite $\P$-a.s. On the other hand, it is clear that $B_1> 0$, $\mathbb{P}$-a.s., and by the scaling property of $Z_t$, we can see that $B_t$ has the same law as $t^{1-\beta} B_1$.
Therefore for any $M>0$, \[\mathbb{P}(B_t\le M)= \mathbb{P}(B_1\le t^{\beta-1}M)\rightarrow 0 \ \textnormal{as}\  t\rightarrow \infty,\]
recalling that $\beta\in(0,1)$. We then conclude $B_t\rightarrow \infty$, $\mathbb{P}$-a.s., as $t\rightarrow\infty$. Let $\gamma_t$ be the inverse of $B_t$. Define \[Y_t= \int_0^t\phi(Z_{s-})\ dZ_s,\ \ V_t= Y_{\gamma_t}\ \ \textnormal{and}\ \ X_t= Z_{\gamma_t}.\] Then by Theorem 3 in \cite{KS},
$V_t$ has the same law as $Z_t$. Some stochastic calculus shows that
\begin{align}
dX_t= \phi(X_{t-})\ dV_t,\ \ X_0= 0.
\label{CE07}
\end{align}

Notice that $X_t$ is not identically zero, while the identically zero process is another solution to \eqref{CE07}. Therefore pathwise uniqueness does not hold.\\

This example also shows that it is not true that there is necessarily
uniqueness to the ordinary differential equation
\[dy(x)= \phi(y(x-))\, z(dx),\]
even when $z$ is a positive purely atomic measure and $\phi$ is non-negative
and non-decreasing.
\end{remark}


\ \\ \ \\

\ni {\bf Hua Ren}\\
Department of Mathematics\\
University of Connecticut \\
Storrs, CT 06269-3009, USA\\
{\tt hua.ren@uconn.edu}
\ms

\end{document}